\documentclass[11pt,oneside]{amsart}
\usepackage{amsmath,amssymb}
\usepackage{bbm}
\usepackage{graphicx,tikz,mathtools}
\usepackage{geometry} 
\usepackage[pdftex]{hyperref} 
\usepackage{cite}
\usepackage{mathrsfs} 
\usepackage{caption}
\hypersetup{
	colorlinks=true,
	linkcolor=blue, 
	citecolor=blue,
	filecolor=blue,
	urlcolor=blue,
}

\makeatletter
\@namedef{subjclassname@2020}{\textup{2020} Mathematics Subject Classification}
\makeatother

\newtheorem{theorem}{Theorem}

\newtheorem{lemma}[theorem]{Lemma}
\newtheorem*{conjecture*}{Conjecture}

\newcommand{\triple}[1]{{\left\vert\kern-0.25ex\left\vert\kern-0.25ex\left\vert #1
        \right\vert\kern-0.25ex\right\vert\kern-0.25ex\right\vert}}

\theoremstyle{definition}
\newtheorem{definition}[theorem]{Definition}
\newtheorem{remark}[theorem]{Remark}

\newcommand{\ep}{\varepsilon}

\newcommand{\te}{\theta}
\newcommand{\vp}{\varphi}

\newcommand{\Om}{\Omega}

\def\NN{\mathbb{N}}
\def\RR{\mathbb{R}}

\def\TT{\mathbb{T}}

\newcommand{\cB}{{\mathcal I}}
\newcommand{\cC}{{\mathcal C}}

\newcommand{\cR}{{\mathcal R}}

\newcommand{\pd}{\partial}
\newcommand\minus\backslash

\newcommand\lan\langle
\newcommand\ran\rangle

\newcommand{\supp}{\operatorname{supp}}

\DeclareMathOperator\Div{div}

\renewcommand\leq\leqslant
\renewcommand\geq\geqslant

\numberwithin{equation}{section}

\begin{document}

\title[Uniformly rotating Euler flows with compact support]{Uniformly rotating Euler flows\\ with compactly supported velocity}

\author[A. Enciso]{Alberto Enciso}
\address{   
\newline
\textbf{{\small Alberto Enciso}} 
\vspace{0.15cm}
\newline \indent Instituto de Ciencias Matem\'aticas, Consejo Superior de Investigaciones Cient\'\i ficas, 28049 \newline \indent Madrid,
 Spain}
\email{aenciso@icmat.es}

 \author[A. J. Fern\'andez]{Antonio J.\ Fern\'andez}
 \address{ \vspace{-0.4cm}
\newline 
\textbf{{\small Antonio J. Fern\'andez}} 
\vspace{0.15cm}
\newline \indent Departamento de Matem\'aticas, Universidad Aut\'onoma de Madrid, 28049 Madrid, Spain}
 \email{antonioj.fernandez@uam.es}

 \author[D. Ruiz]{David Ruiz}
 \address{ \vspace{-0.4cm}
\newline 
\textbf{{\small David Ruiz}} 
\vspace{0.15cm}
\newline \indent  
 IMAG, Departamento de An\'alisis Matem\'atico, Universidad de Granada, 18071 Granada, \newline \indent Spain}
 \email{daruiz@ugr.es}



%
%
\begin{abstract}  
For any positive integer $k$, we prove the existence of nontrivial $C^k$-smooth uniformly rotating solutions to the 2D incompressible Euler equations with compact spatial support. These solutions, which can be chosen to be small perturbations of radial flows, are the first example of smooth rotating flows with finite energy which are not locally radial. We also prove new rigidity results for rotating solutions which show that the geometric structure of these flows is severely constrained.

\end{abstract}

\maketitle

\section{Introduction}
Let us consider the 2D incompressible Euler equations,
\begin{equation}\label{E.Euler} 
\left\{
    \begin{aligned}
        \ &\pd_t v  + v\cdot \nabla v+\nabla p=0 \quad && \textup{in } \RR^2 \times (0,T)\,, \\
        & \Div v = 0 && \textup{in } \RR^2 \times (0,T)\,,\\
        &  v(\cdot,0 ) = v_0 && \textup{in } \RR^2\,,
    \end{aligned}
    \right.
\end{equation}
which describe the evolution of an ideal planar fluid in terms of its initial velocity, given by a divergence-free vector field $v_0:\RR^2\to\RR^2$. This system of equations can be conveniently written in terms of the vorticity $\omega:= \nabla^\perp\cdot v$ as
\begin{equation*}
\left\{
    \begin{aligned}
        \ &\pd_t \omega  + v\cdot \nabla \omega=0 \quad && \textup{in } \RR^2 \times (0,T)\,, \\
        &  \omega(\cdot,0 ) = \omega_0 && \textup{in } \RR^2\,,
    \end{aligned}
    \right.
\end{equation*}
where $v=\nabla^\perp \Delta^{-1}\omega$ is regarded as a nonlocal function of~$\omega$, and where $\omega_0:=\nabla^\perp\cdot v_0$. Here and in what follows, $\nabla^{\perp} := (-\pd_{x_2},\, \pd_{x_1})$. Once~$v$ is known, the pressure can be subsequently recovered by inverting an elliptic equation.

In this paper, we are concerned with {\em uniformly rotating}\/ solutions, that is, solutions of the form
\begin{equation}\label{E.rotating} 
v(x,t)= \cR_{-\Omega t} v_0(\cR_{\Omega t} x)\,,
\end{equation}
where $\cR_{\alpha}\in \mathrm{Mat}_{2\times 2}$ denotes the rotation matrix of angle $\alpha\in\RR$ in the clockwise direction, and where $\Om \in \RR$ is the (constant) {\em angular velocity}\/ of the solution. Equivalently, this means that the vorticity satisfies
$$ \omega(x,t) = \omega_0(\cR_{ \Omega t} x)\,.$$
It is easy to see that the initial velocity~$v_0$ generates a solution of this form if and only if
\begin{equation}\label{E.Euler1} 
(v_0 -  \Omega x^\perp) \cdot \nabla \omega_0 =0  \quad \textup{in } \RR^2 \,.
\end{equation}

There is a large literature on rotating Euler flows, mainly in the case of patches (that is, solutions whose initial vorticity is an indicator function). The first nontrivial example was found by Kirchhoff~\cite{Kirchhoff1876}, who showed that elliptical patches rotate uniformly. Later, Deem and Zabusky~\cite{DeemZabusky1978} identified numerically families of simply-connected rotating patches with $m$-fold symmetry, a result rigorously established by Burbea~\cite{Burbea1982} through local bifurcation from the unit disk for each $m\ge2$. Hmidi, Mateu, and Verdera~\cite{HmidiMateuVerdera2013} and Castro, C\'ordoba, and G\'omez-Serrano~\cite{CastroCordobaGomezSerrano2016} established the boundary regularity of these solutions, and Hassainia, Masmoudi, and Wheeler~\cite{HassainiaMasmoudiWheeler19} recently studied global bifurcation properties that are consistent with the formation of corner-like structures. Further information on uniformly rotating patches can be found in~\cite{HmidiMateu2017, HmidiMateuVerdera2015, Park2022} and references therein.

The only known nontrivial smooth rotating solutions with compactly supported vorticity were constructed by Castro, C\'ordoba, and G\'omez-Serrano in~\cite{CastroCordobaGomezSerrano2019}. This is proved via bifurcation from radial solutions, with the key idea of writing the vorticity in terms of a modified stream function so that each level curve deforms a circular contour of the radial profile. Rotating solutions with discontinuous compactly supported vorticity (different from indicator functions) were constructed in \cite{GarciaHmidiSoler2020,GarciaHmidiMateu2024}. Just as in the case of patches, in all these examples the velocity decays as $1/|x|$ at infinity, so in particular the energy is not finite.

In this paper we are concerned with solutions for which the velocity field is compactly supported. Since the velocity is divergence-free, we can write it in terms of a stream function as $v_0=\nabla^\perp\psi$. We now use the change of variables $\phi(x):=\psi(x) - \frac{\Om}{2}|x|^2$ so that \eqref{E.Euler1} becomes
\begin{equation}\label{E.Euler2} 
	\nabla^\perp\phi  \cdot \nabla \Delta\phi =0  \quad \textup{in } \RR^2\,.
\end{equation}
Note that the velocity has compact spatial support if and only if $\phi(x)+\frac\Om2|x|^2$ is constant outside a ball. A trivial class of solutions to \eqref{E.Euler2} are the {\em radial}\/ ones: any function~$\phi$ which only depends on the distance to the origin satisfies~\eqref{E.Euler2} for any angular velocity~$\Omega$, and the solution is in fact stationary. A slightly more general class of solutions, which we call {\em locally radial}\footnote{For details, see Definition~\ref{D.locradial} in the main text.}, is obtained by taking $\phi$ to be a superposition of functions which are radially symmetric around different points. If we glue these functions so that  $\phi(x) = - \frac{\Om}{2} |x|^2$ outside a ball, then we obtain a compactly supported rotating solution (see Figure \ref{F.fig1}).

\begin{figure}[t] 
	\centering 
	\begin{minipage}[c]{120mm}
			\centering
			\resizebox{120mm}{60mm}{\includegraphics{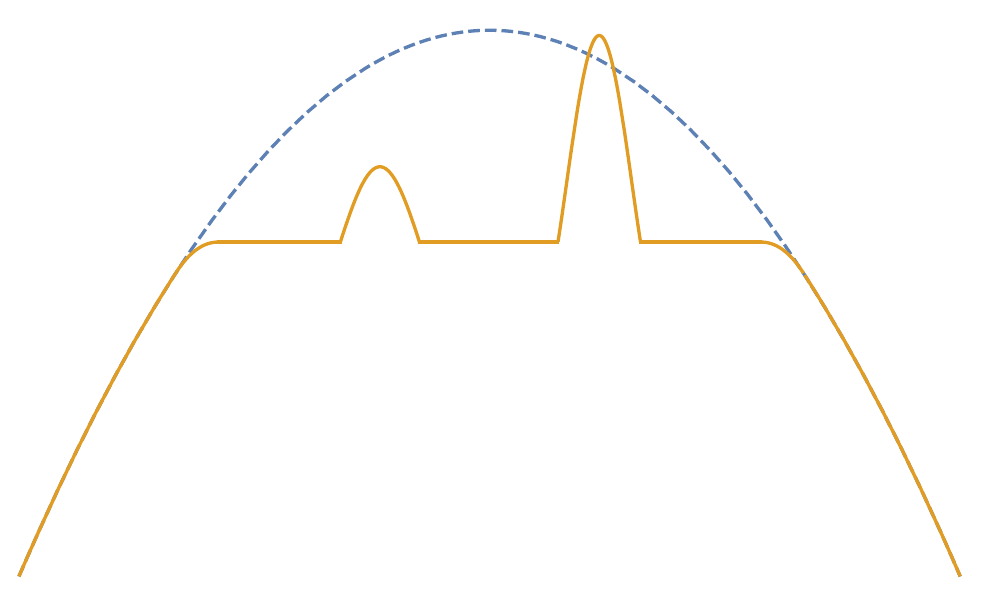}}
		\end{minipage}
	\caption{An illustrative diagram of the graph of a locally radial, but nonradial, function $\phi$ satisfying \eqref{E.Euler2}. Observe that $\phi(x)$ coincides with $- \frac{\Omega}{2}|x|^2$ (dashed) outside a ball, so the velocity has compact support.}
    \label{F.fig1}
\end{figure}

Our objective is to show the existence of compactly supported rotating solutions, with any angular velocity, for which $\phi$ is not locally radial. Specifically, our first main result is the following.

\begin{theorem}\label{T.main1}
For any positive integer $k$ and any $\Omega \in \RR$, there exist uniformly rotating Euler flows with angular velocity $\Omega$  whose initial velocity belongs to $C^k_c(\RR^2)$, and which are not locally radial. 
\end{theorem}

The proof of Theorem \ref{T.main1} relies on the existence of $C^k$~compactly supported stationary Euler flows, which we proved in our previous work~\cite{EFR}. Roughly speaking, the idea is to glue these stationary solutions to the exterior radial function $-\frac{\Om}{2}|x|^2$. Details are given in Section~\ref{S.Flexibility}.  

One might consider this to be an unusual approach to construct compactly supported, uniformly rotating solutions to \eqref{E.Euler}. However, the construction admits very limited flexibility: the geometric structure of these kind of flows is in fact subject to stringent constraints. The next two results make these restrictions explicit.

First, we show that this gluing procedure is essentially the only way to proceed, in that the gradient of the function~$\phi$ must vanish on any circle where it ceases to be radial.

\begin{theorem} \label{T.Rigidity2}
    Let $\phi \in C^3(\RR^2)$ be a solution to~\eqref{E.Euler2} such that $\phi(x) = - \frac{\Om}{2}|x|^2$ outside a ball. Consider the nonempty closed set 
    $$
    \cB:= \{r\in[0,\infty): \phi(x) = \phi(y) \textup{ for all } x, y \in \RR^2 \textup{ with } |x| = |y| = r\}.
    $$ 
    If $r \in \partial \cB\backslash\{0\}$ is not isolated in $\cB$, then $\nabla \phi = 0$ on $\pd B_r(0)$.
\end{theorem}

This result is quite different from the previous rigidity results for uniformly rotating Euler flows. Among them, let us mention that G\'omez-Serrano, Park, Shi, and Yao~\cite{GPSY} showed that any uniformly rotating smooth solution with nonnegative vorticity and compact support is radial whenever $\Omega< 0$ (and locally radial if $\Omega=0$). Fan, Wang, and Zhang~\cite{FanWangZhan25} recently strengthened this sufficient condition to $\Omega \in (-\infty,\,\tfrac{1}{2}\inf \omega] \cup [\tfrac{1}{2}\sup \omega,\,\infty)$, both in the patch and in the smooth settings.

Finally, let us stress that the rotating solutions constructed in Theorem \ref{T.main1} do {\em not}\/ satisfy an equation of the form
\begin{equation}\label{E.semilinearphi}
	\Delta \phi + f(\phi ) = 0 \quad \textup{in } \RR^2\,,
\end{equation}
with $f\in C(\RR)$. We conclude the paper showing that this aspect of the construction is also necessary. In this direction, note that there is a large literature on rigidity results proving radial symmetry of solutions to 2D Euler equations in the stationary case; see, for instance, \cite{GPSY, R, Hamel, DeR-E-R, wang}, and the references therein. In the next theorem we prove radial symmetry for rotating solutions with compact support if $\phi$ is a solution to \eqref{E.semilinearphi}. 

\begin{theorem} \label{T.Rigidity1}
     Given some $\Om \in \RR \setminus \{0\}$, and some $f\in C(\RR)$, let $\phi \in C^{1}(\RR^2)$ be a solution to~\eqref{E.semilinearphi} such that
     $\phi(x)= - \frac{\Om}{2}|x|^2 
     $
     outside a ball. Then $\phi$ is locally radial. If, in addition, $f \in C^1(\RR)$, then $\phi$ is radially symmetric with respect to the origin.
\end{theorem}

\section{Flexibility results} \label{S.Flexibility}

This section proves the existence of smooth compactly supported uniformly rotating Euler flows which are not locally radial, as stated in Theorem~\ref{T.main1}. 

Before going on, let us recall (see e.g.~\cite[Section 6]{Brock}) the precise definition of a locally radial function:

\begin{definition}\label{D.locradial} 
A function $\phi\in C^1(\RR^2)$ is {\em locally radial} if $\RR^2= S \cup A$, with $A= \bigcup_{k\in K}A_k$,
	where:
	\begin{enumerate}
		\item $ S=\{x\in\RR^2:\nabla \phi(x)=0\}$ is the critical set of~$\phi$.
		\item $K$ is a countable set,
		\item The sets $A_k=\{x\in\RR^2: r_k<|x-q_k|<R_k\}$ are disjoint open annuli, with $0 \leq r_k < R_k \leq +\infty$ and $q_k\in\RR^2$, and the restriction $\phi|_{A_k}$  depends only on the distance to the point~$q_k$.
	\end{enumerate}
\end{definition}

Since the proof of Theorem \ref{T.main1} relies on the smooth compactly supported stationary solutions to \eqref{E.Euler} we constructed in \cite{EFR}, let us now recall their main properties. Given a constant $a \geq {4}$ and functions $b,B\in C^{k+1}(\TT)$ bounded, for example, as
\[
\|b\|_{L^\infty(\TT)}+\|B\|_{L^\infty(\TT)}<\tfrac{1}{ 10}\,,
\]
we consider bounded domains defined in polar coordinates by
\begin{equation}\label{E.defOm}
\Om_a^{b,B}:=\{(r,\te)\in\RR^+\times\TT: a-3+b(\te)< r< a+3+B(\te)\}\,.
\end{equation}

\noindent The main result in \cite{EFR}, namely \cite[Theorem 1.1]{EFR}, is obtained as a consequence of the following elliptic result:

\begin{theorem}{\rm \hspace{-0.08cm} (\hspace{-0.1cm} \cite[Theorem 2.2]{EFR})} \label{T.EllipticStationary} There exist sequences $(a_n)_{n=1}^\infty \subset [4,+\infty)$, $(b_n)_{n=1}^\infty, (B_n)_{n=1}^\infty \subset C^{k+1}(\TT)$, and $(\phi_n)_{n=1}^\infty \subset C^{k+1}(\overline{\Omega}_{a_n}^{\,b_n,B_n})$ such that:
\begin{enumerate}
	\item $a_n\to a^*$ as $n\to\infty$. \smallbreak
	\item $b_n$ and $B_n$ are nonconstant functions which tend to~$0$ in the $C^{k+1}$-norm. \smallbreak
	\item $\phi_n$ are positive solutions to	
\begin{equation}\label{Ecuacion}
	\Delta \phi_n + f_{a_n}(|x|,\, \phi_n)  = 0\quad \text{ in }\Om_{a_n}^{\,b_n,B_n} \,,
\end{equation}
such that $\mathrm{D}^m\phi_n$ vanish identically on~$\partial \Omega_{a_n}^{b_n,B_n}$ for all $0\leq m\leq k+1$. \smallbreak

\item $\nabla^{\perp} \phi_n \cdot \nabla \Delta\phi_ n = 0$ in $\Omega_{a_n}^{b_n,B_n}$.  

\end{enumerate} 
\end{theorem}

\begin{remark}
    Note that the functions $f_{a_n}$ are only H\"older continuous in the second variable and cannot be freely chosen. We refer to \cite[Sections 2 and 3]{EFR} for the precise construction of these functions. 
\end{remark}

Having this theorem at hand, we set 
$$
\overline{\phi}_n(x) := \left \{
\begin{aligned}
\ & \phi_n(x) \quad && \textup{if } x \in \Om_{a_n}^{b_n,B_n}\,,\\
& 0 && \textup{if } x \not \in  \Om_{a_n}^{b_n,B_n}\,,
\end{aligned}
\right.
$$
which, by construction, is not locally radial. By \cite[Section 2]{EFR},  we get the existence of $n_0 \in \NN$ such that, for all $n \in \NN$ with $n \geq n_0$,
\begin{equation} \label{E.vnStationary}
\overline{v}_n := \nabla^{\perp} \overline{\phi}_n\,,
\end{equation}
is a compactly supported stationary Euler flow of class $C^k(\RR^2)$, which is not locally radial. In particular, for all $n \geq n_0$, $\overline{\phi}_n \in C^{k+1}(\RR^2)$ satisfies
$$
\nabla^{\perp} \overline{\phi}_n \cdot \nabla \Delta\overline{\phi}_n = 0 \quad \textup{in } \RR^2\,,
$$
and
$$
\supp(\overline{\phi}_n) \subseteq \overline{\Omega}_{a_n}^{\,b_n, B_n}\,.
$$

Next, we can prove our flexibility result, namely Theorem \ref{T.main1}.

\begin{proof}[Proof of Theorem \ref{T.main1}]
Let $\Om \in \RR$ be fixed but arbitrary, and let $n_0 \in \NN$ be as in \eqref{E.vnStationary}. We choose and fix $R := R(n_0)>0$ sufficiently large so that, for all $n \geq n_0$, 
$$
\overline{\Omega}_{a_n}^{\,b_n,B_n} \subset \subset B_{R}(0)\,.
$$
We consider a cut-off function $\chi \in C^\infty(\RR)$ such that
$$
\chi(t) = 1 \quad \textup{if }|t| \leq 1\,, \quad \textup{and} \quad \chi(t) = 0 \quad \textup{if } |t| \geq 2\,, 
$$
and, for all $n \geq n_0$, we set
$$
\widetilde{\phi}_n (x) := \chi\bigg(\frac{|x|}{R} \bigg) \overline{\phi}_n(x) - \left( 1- \chi\bigg(\frac{|x|}{R} \bigg) \right) \frac{\Om}{2}|x|^2\,.
$$

At this point, it is straightforward to check that, for all $n \geq n_0$,
$$
v_{0,n} := \nabla^{\perp} \widetilde{\phi}_n + \Om x^{\perp}\,,
$$
defines a compactly supported uniformly rotating solution to \eqref{E.Euler} of class $C^k(\RR^2)$ with angular velocity $\Om$, which is not locally radial. This concludes the proof. 
\end{proof}

\begin{remark}
Note that the corresponding vorticity is given by
$$
\omega_{0,n} := \nabla^{\perp} \cdot v_{0,n} = \Delta \widetilde{\phi}_n + 2 \Omega\,.
$$
In particular, it follows that
$$
\omega_{0,n}(x) = f_{a_n}(|x|, \overline{\phi}_n(x)) + 2\Om \quad \textup{in } B_R(0)\,,
$$
or equivalently that
\begin{equation} \label{E.remarkJavi}
\frac12\,\omega_{0,n}(x) - \Omega = f_{a_n}(|x|,\overline{\phi}_n(x)) \quad \textup{in } B_R(0)\,.
\end{equation}
Since $f_{a_n}$ changes sign (see \cite[Lemma 3.1 and Figure 1]{EFR}), if we want $\omega_{0,n}$ to be non-negative, we need to assume that $\Omega > 0$. This agrees with the rigidity results by G\'omez-Serrano, Park, Shi and Yao (see \cite[Theorem B]{GPSY}). Furthermore, from this change of sign and \eqref{E.remarkJavi}, we infer that
$$
\inf_{\RR^2} \omega_{0,n} \leq \inf_{B_R(0)} \frac12\,\omega_{0,n} < \Om < \sup_{B_R(0)} \frac12\,\omega_{0,n} \leq \sup_{\RR^2} \omega_{0,n} \,,
$$
which agrees with the recent paper by Fan, Wang and Zhang (see \cite[Theorem 1.4]{FanWangZhan25}). 
\end{remark}

\section{Rigidity results}

In this section we prove the two rigidity theorems stated in the Introduction.

First, let us state an auxiliary lemma. Although it is well known, we include a short proof for completeness.

\begin{lemma} \label{L.semilinear}
    Let $\Om \subseteq \RR^2$ be an open set and assume that $f,g \in C^1(\Omega)$ satisfy that
    $$
    \nabla^{\perp} f \cdot \nabla g = 0 \quad \textup{in } \Omega\,.
    $$
    If there exists $x_* \in \Omega$ such that $\nabla f (x_*) \neq 0$, then there exist a neighborhood $\mathcal U \subset \Omega$ of $x_*$ and a $C^1$ function $F: \RR \to \RR$ such that
    \begin{equation} \label{E.conclussionL.semilinear}
    g(x) = F(f(x)) \quad \textup{for all } x \in \mathcal{U}\,.
    \end{equation}
\end{lemma}

\begin{proof}
    Since $\nabla f(x_*) \neq 0$, up to a rotation if necessary, we can assume that $\pd_{x_1} f(x_*) \neq 0$. We then introduce $$
    \Psi: \Om \to \RR^2\,, \quad x \mapsto(f(x_1,x_2),x_2)\,,
    $$
    and stress that
    $$
    \det ({\rm D} \Psi(x_*)) = \pd_{x_1} f(x_*) \neq 0\,.
    $$
    By the inverse function theorem, there exist neighborhoods $\mathcal{U} \subset \Omega$ of $x_*$, and $\mathcal{V}$ of $f(x_*)$ such that $f(\mathcal{U}) \subset \mathcal{V}$, and $\Psi: \mathcal{U} \to \mathcal{V}$ is a diffeomorphism. In particular, $y = \Psi(x)$ defines a $C^1$ system of local coordinates in $\mathcal{U}$. Moreover, taking $\mathcal{U}$ smaller if necessary, we can assume that $\pd_{x_1} f \neq 0$ in $\mathcal{U}$.

    We now set $\widetilde{g} := g \circ \Psi^{-1}: \mathcal{V} \to \RR$ and observe that
    \begin{align*}
        \pd_{y_2} \widetilde{g}(y) & = - \pd_{x_1} g (\Psi^{-1}(y)) \frac{\pd_{x_2}f(\Psi^{-1}(y))}{\pd_{x_1}f(\Psi^{-1}(y))} + \pd_{x_2} g(\Psi^{-1}(y)) \\
        & = \frac{1}{\pd_{x_1}f(\Psi^{-1}(y))} \nabla^{\perp}f(\Psi^{-1}(y)) \cdot \nabla g (\Psi^{-1}(y))= 0\,.
    \end{align*}
    Hence, there exists a $C^1$ function $F: f(\mathcal{U}) \to \RR$ such that $\widetilde{g}(y) = F(y_1)$. We can then take any $C^1$~extension of this function, which we denote by $F:\RR\to\RR$ with some abuse of notation. Going back to the original coordinates, we conclude that \eqref{E.conclussionL.semilinear} holds.
\end{proof}

\begin{proof}[Proof of Theorem \ref{T.Rigidity2}] 
    Let $r \in \pd \cB\backslash\{0\}$ be non-isolated in $\cB$. In view of Definition~\ref{D.locradial}, we argue by contradiction, assuming that there exists $x_* \in \pd B_r(0)$ such that $\nabla \phi(x_*) \neq 0$ and,  for any $x \in \RR^2 \setminus \{0\}$, we denote by $\nu_x = \frac{x}{|x|}$ the normal vector to $\partial B_{|x|}(0)$. 

    First of all, we prove that $\frac{\pd \phi}{\pd \nu_x} (x) = \frac{\pd \phi}{\pd \nu_y} (y) \neq 0$ for all $x$, $y$ with $|x|=|y|=r$. 
    
    Since $r > 0$ is not isolated, we get the existence of $(r_n)_n \subset \cB$ such that $r_n \to r$ as $n \to \infty$. Also, we set $t_n := r_n -r$ for all $n \in \NN$, and stress that $t_n \to 0$ as $n \to \infty$. Then, we infer that
    $$
    \phi(x+ t_n \nu_x) = \phi(y + t_n \nu_y)\,, \quad \textup{ for all } x,y \in \pd B_r(0)\,.
    $$
    By continuity,
    $$
    \phi(x) = \phi(y) \quad \textup{for all }x,y \in \pd B_r(0)\,.
    $$ In other words, we have shown that the $ r \in \cB$ (more generally, the set $\cB$ is closed). Moreover, for any $x,\,y \in \pd B_r(0)$,
    $$
    \frac{\pd \phi}{\pd \nu_x} (x) = \lim_{t \to 0} \frac{\phi(x+ t \nu_x)-\phi(x)}{t} = \lim_{t \to 0} \frac{\phi(y+ t \nu_y)-\phi(y)}{t} = \frac{\pd \phi}{\pd \nu_y} (y)\,.
    $$
    Combining that $r \in \cB$, the assumption $\nabla \phi(x_*) \neq 0$ and the identity above, we conclude that
    $$
    \frac{\pd \phi}{\pd \nu_x} (x) = \frac{\pd \phi}{\pd \nu_y} (y) \neq 0 \quad \textup{for all  } x,\, y\textup{ with }|x|=|y|=r\,,
    $$
    as desired.

    Now, let $x \in \pd B_r(0)$ be fixed but arbitrary. Since $\phi$ is a solution to $\nabla^{\perp}\phi \cdot \nabla \Delta \phi = 0$ in $\RR^2$ and $\nabla \phi(x) \neq 0$, by Lemma \ref{L.semilinear}, there exist $\ep > 0$ and $f \in C^1(\overline{B}_\ep(x))$ such that
    $$
    \Delta \phi + f(\phi) = 0 \quad \textup{in } B_\ep (x)\,.
    $$
    Having this function $f$ at hand, and taking $\ep > 0$ smaller if necessary, we denote by $\varphi$ the unique solution to the initial value problem:
    $$
    \varphi'' + \frac{1}{r} \varphi' + f(\varphi) = 0 \quad \textup{in } (r-\ep,r+\ep)\,, \quad \vp(r) = \phi(r)\,, \quad \vp'(r) = \phi'(r)\,.
    $$
    Here we are denoting by $\phi(r)$ and $\phi'(r)$ the values $\phi(x)$ and $\frac{\pd \phi}{\pd \nu_x} (x)$, respectively, for any choice $x \in \partial B_r(0)$. By the previous argument, these quantities do not depend on $x$.

    Then, we have that $\phi$ and $\vp$ are both solutions to 
    $$
    \Delta \psi + f(\psi) = 0 \quad \textup{in } B_\ep(x)\,, \quad \psi = \vp \,, \quad \frac{\pd \psi}{\pd \nu} = \vp' \quad \textup{on } \pd B_r(0)\,.
    $$
    By unique continuation, we then conclude that $\phi \equiv \vp$ in $B_\ep(x)$, which in turn implies that $\phi$ is radial in $B_\ep(x)$. 

    Since the point $x \in \pd B_r(0)$ was fixed but arbitrary, we can repeat this argument for every $x \in \pd B_r(0)$ and obtain that, for every $x \in \pd B_r(0)$ there exists $\ep_x > 0$ such that $\phi$ is radial in $B_{\ep_x}(x)$. Also, note that $\{B_{\frac{\ep_x}{2}}(x): x \in \pd B_r(0)\}$ is an open covering of $\pd B_r(0)$. Hence, we can cover $\pd B_r(0)$ by finitely many such balls. That is, there exist $x_1, \ldots, x_k \in \pd B_r(0)$ such that
    $$
    \pd B_r(0) \subset \bigcup_{\ell=1}^k B_{\frac{\ep_{\ell}}{2}}(x_\ell)\,.
    $$
    We use here the notation $\ep_\ell := \ep_{x_\ell}$ for all $\ell$. Then, setting $\widetilde{\ep} := \min_{\ell} \tfrac{\ep_\ell}{2}$, we get that
    $$
    \cC:= \{x \in \RR^2: r-\widetilde{\ep} < |x| < r+ \widetilde{\ep}\, \} \subset \bigcup_{\ell=1}^k B_{\ep_\ell}(x_\ell)\,,
    $$
    and so that $\phi$ is radial in $\cC$. This implies that $r \in \mathring{\cB}$, which is a contradiction.  
\end{proof}

Let us now present the proof of the remaining rigidity result, namely Theorem \ref{T.Rigidity1}:

\begin{proof}[Proof of Theorem \ref{T.Rigidity1}]
    We assume that $\Om > 0$. The result in the case where $\Om < 0$ follows from the one with $\Om > 0$ upon replacing $\phi$ by $-\phi$.  We set $m := \inf_{\overline{B}_R(0)} \phi$, and stress that, since $\phi$ is continuous, the infimum actually corresponds to a minimum. Then, we choose $\widetilde{R} \geq R$ sufficiently large so that
    $$
    m > - \frac{\Om}{2}\,\widetilde{R}^{2}\,,
    $$
    and we get that
     $$
     \Delta \phi + f(\phi ) = 0 \quad \textup{in } \RR^2\,, \quad \phi = - \frac{\Om}{2}|x|^2 \quad \textup{in } \RR^2 \setminus B_{\widetilde{R}}(0)\,, $$
     and
     $$
     \phi(x) > - \frac{\Om}{2} \widetilde{R}^{2} \quad \textup{in } B_{\widetilde{R}}(0)\,.
     $$
     
     At this point, we set
     $$
     \varphi(x) := \phi(x) + \frac{\Om}{2} \widetilde{R}^{2}\quad \textup{and} \quad g(s):= f\Big(s- \frac{\Om}{2} \widetilde{R}^{2}\Big)\,,
     $$
     and stress that $\varphi$ is a solution to
     \begin{equation*} 
     \Delta \varphi + g(\varphi) = 0\,, \quad \varphi > 0 \quad \textup{in } B_{\widetilde{R}}(0)\,, \quad \varphi = 0 \quad \textup{on } \pd B_{\widetilde{R}}(0)\,.
     \end{equation*}
     Having this equation at hand, the result immediately follows from  \cite[Theorem 7.2 and Corollary 7.6]{Brock} in the case where $f \in C(\RR)$, and \cite[Theorem 1]{GNN} in the case where $f \in C^1(\RR)$.
\end{proof}

\section*{Acknowledgments}  
This work has received funding from the European Research Council (ERC) under the European Union's Horizon 2020 research and innovation programme through the grant agreement~862342 (A. Enciso). A. Enciso is also partially supported by the grant PID2022-136795NB-I00 of the Spanish Science Agency, and the ICMAT--Severo Ochoa grant CEX20 19-000904-S. A. J. Fern\'andez is partially supported by the grants PID2023-149451NA-I00 of MCIN/AEI/10.13039/ 501100011033/ FEDER, UE, and Proyecto de Consolidaci\'on Investigadora 2022, CNS2022-135640, MICINN (Spain). D. Ruiz has been supported by: the Grant PID2024-155314NB-I00 of the MICIN/AEI, the IMAG-Maria de Maeztu Excellence Grant CEX2020-001105-M funded by MICIN/AEI, and the Research Group FQM-116 funded by J. Andaluc\'ia.

\bibliographystyle{amsplain}

\end{document}